\documentclass[12pt]{article}
\usepackage{amssymb}
\usepackage{amsfonts}
\usepackage{amsmath}
\usepackage[usenames]{color}
\usepackage{mathrsfs}
\usepackage{amsfonts}
\usepackage{amssymb,amsmath}
\usepackage{CJK}
\usepackage{cite}
\usepackage{cases}
\usepackage{amsthm}

\def\cl{\centerline}

\def\vs{\vspace*}
\def\S{\mathfrak{S}}
\def\L{\mathscr{W}}

\def\Z{\mathbb{Z}}
\def\N{\mathbb{N}}

\def\C{\mathbb{C}}

\def\ni{\noindent}

\numberwithin{equation}{section}
\newtheorem{theo}{Theorem}[section]

\newtheorem{coro}[theo]{Corollary}
\newtheorem{lemm}[theo]{Lemma}
\newtheorem{prop}[theo]{Proposition}

\newtheorem{case}{Case}

\newtheorem{exam}[theo]{Example}
\newtheorem{rema}[theo]{Remark}

\newtheorem{remark}[theo]{Remark}

\begin{document}
\begin{center}
\cl{\large\bf \vs{8pt} A class of non-weight modules  over  the super-BMS$_3$ algebra}
\cl{ Haibo Chen, Xiansheng Dai,  Ying Liu and Yucai Su}
\end{center}

{\small
\parskip .005 truein
\baselineskip 3pt \lineskip 3pt

\noindent{{\bf Abstract:}
In the present paper,  a class of non-weight modules over the super-BMS$_3$ algebras $\S^{\epsilon}$ ($\epsilon=0$ or $\frac{1}{2}$) are constructed.
Assume that   $\mathfrak{t}=\C L_0\oplus\C W_0\oplus\C G_0$
and $\mathfrak{T}=\C L_0\oplus\C W_0$ are  the  Cartan subalgebra (modulo center) of $\S^{0}$ and $\S^{\frac{1}{2}}$, respectively.  These modules over $\S^{0}$  when
restricted to the     $\mathfrak{t}$ are free of rank $1$, while
these modules over $\S^{\frac{1}{2}}$  when
restricted to the $\mathfrak{T}$   are free of rank $2$.
Then we determine the necessary and
sufficient conditions for these   modules being simple, as
well as determining the necessary and sufficient conditions
for two  $\S^{\epsilon}$-modules being isomorphic.
%Moreover,  we see that the category of free $U(\mathfrak{t})$-modules of rank $1$ over  $\S^0$ is
%equivalent to the category of free $U(\mathfrak{T})$-modules of rank $2$ over
%$\S^{\frac{1}{2}}$.
\vs{5pt}

\ni{\bf Key words:}
super-BMS$_3$ algebra, non-weight module, simple module.}

\ni{\it Mathematics Subject Classification (2020):} 17B10, 17B65, 17B68.}
\parskip .001 truein\baselineskip 6pt \lineskip 6pt
\section{Introduction}

Throughout this paper,  we denote by $\C,\C^*,\Z,\Z_+$ and  $\N$  the sets of complex numbers, nonzero complex numbers,  integers, nonnegative integers and positive integers, respectively.
   All vector superspaces
(resp. superalgebras, supermodules)   and
 spaces (resp. algebras, modules)    are considered to be over
the field  of complex numbers.

The   \textit{super-BMS$_3$ algebra}  $\mathfrak{S}^{\epsilon}$ ($\epsilon=0$ or $\frac{1}{2}$)
 is defined as an infinite-dimensional Lie superalgebra over $\C$ with basis
$\{L_m,W_m,G_{r},C_1,C_2
\mid m\in \Z, r\in \epsilon+\Z\}$ and  satisfying the following  non-trivial relations
 \begin{equation}\label{def1.1}
\aligned
&[L_m,L_n]= (n-m)L_{m+n}+\frac{n^{3}-n}{12}\delta_{m+n,0}C_1,\\&
[L_m,W_n]= (n-m)W_{m+n}+\frac{n^{3}-n}{12}\delta_{m+n,0}C_2,\\&
[L_m,G_r]= (r-\frac{m}{2})G_{m+r},\\&
 [G_r,G_s]= 2W_{r+s}+\frac{1}{3}(r^2-\frac{1}{4})\delta_{r+s,0}C_2,
\endaligned
\end{equation}
  where $m,n\in\Z,\epsilon=0,\frac{1}{2}, r,s\in\epsilon+\Z$ and $C_1,C_2$ are central elements.
 In order to study  the
  boundary
conditions for the    behaviour of   gauge fields, the super-BMS$_3$ algebra was introduced in \cite{BDMT1,BDMT2}.
Clearly, the $\Z_2$-grading of  $\mathfrak{S}^\epsilon$ is defined by   $\mathfrak{S}^\epsilon=\mathfrak{S}_{\bar 0}\oplus \mathfrak{S}_{\bar 1}^\epsilon$, where
$\mathfrak{S}_{\bar 0}=\{L_m,W_m,C_1,C_2\mid m\in \Z\}$ and $\mathfrak{S}_{\bar 1}^\epsilon=\{G_r\mid r\in \epsilon+\Z\}$.
  Notice that the subspace spanned by $\{L_m,C_1
\mid m\in \Z\}$ is
the Virasoro algebra, which is  closely related to the quantum physics, conformal field theory, vertex algebras, and so on (see, e.g., \cite{LL,GO,FMS}).
 The even part $\mathfrak{S}_{\bar 0}$ is exactly  the W-algebra   $W(2,2)$ $\mathcal{W}$, which  was introduced in \cite{ZD}  for the study of the classification of vertex
operator algebras generated by weight 2 vectors.
The structure theory  and representation theory   over the super-BMS$_3$ algebra were studied in \cite{DCL,DGL,DGL1,CS1,CS2}.
To illustrate the difference, $\mathfrak{S}^0$ and  $\mathfrak{S}^{\frac{1}{2}}$ are   respectively called  Ramond type super-BMS$_3$ algebra and
Neveu-Schwarz type super-BMS$_3$ algebra. It is clear that $\S^{\frac{1}{2}}$   is isomorphic to the subalgebra of $\S^{0}$ spanned by
  $\{L_{m},W_m\mid m\in2\Z\}\cup\{G_r\mid r\in2\Z+1\}\cup\{C_1,C_2\}$.
%The Cartan subalgebras of   $\S^{0}$ and   $\S^{\frac{1}{2}}$  are   $\mathfrak{h}=\C L_0\oplus\C W_0\oplus\C G_0$
%and $\mathfrak{H}=\C L_0\oplus\C W_0$, respectively.

 As an   important part  of  representation theory, non-weight module has drawn a lot of attention of researchers.
In recent years, a class of    non-weight modules
on which the Cartan subalgebra acts freely were constructed and studied.
This kind of non-weight modules are called free $U(\mathfrak{h})$-modules by many authors,
where $U(\mathfrak{h})$ is the universal enveloping algebra of the Cartan subalgebra $\mathfrak{h}$.
It is worth mentioning that the free $U(\mathfrak{h})$-modules were first constructed  in  \cite{N1}  for the complex matrices algebra $\mathfrak{sl}_{n+1}$.
In addition, these modules were introduced by a very different approach  in \cite{TZ1}  at the same.
From then on, this kind of non-weight modules  were widely investigated, such as   the finite-dimensional simple Lie algebras \cite{N2},     the Virasoro algebra \cite{LZ,TZ}, the Heisenberg-Virasoro algebra \cite{CG,CS}, the algebra $\mathrm{Vir}(a,b)$ \cite{HCS}.
Furthermore,   the super version of  this kind of non-weight modules were also studied.  In \cite{CZ},    $U(\mathfrak{h})$-modules over the basic
Lie superalgebras were investigated, which   showed that $\mathfrak{osp}(1|2n)$ is the only basic Lie superalgebra that
admits such modules.  The free $U(\mathfrak{h})$-modules of rank $1$ over the $N=1$ Ramond  algebra and free
$U(\mathfrak{h})$-modules of rank $2$ over the $N=1$ Neveu-Schwarz algebra were  completely
classified in \cite{YYX1}.  The similar modules over
$N=2$ superconformal algebras were also investigated  in \cite{YYX2,CDL3}.  In this paper, we aim to
study the free $U(\mathfrak{t})$-modules of rank $1$ and $U(\mathfrak{T})$-modules of rank $2$  respectively over the Ramond
type and Neveu-Schwarz type of super-BMS$_3$ algebras.

The rest of this paper is organized as follows.
In section $2$, we    construct a class of non-weight modules over the super-BMS$_3$ algebra  $\S^{\epsilon}$. Then the simplicity and
isomorphism classes of these modules are determined.
In Section $3$,    the free $U(\mathfrak{t})$-modules of rank $1$ over the Ramond type  $\S^0$ are classified.
These modules can   be regarded
as free $U(\mathfrak{T})$-modules of rank $2$. In Section $4$,   the free $U(\mathfrak{T})$-modules of rank
$2$ over the Neveu-Schwarz type  $\S^{\frac{1}{2}}$  are classified. Furthermore, we present that the category of free $U(\mathfrak{t})$-modules
of rank $1$ over the  Ramond type $\S^0$  is equivalent to the category of
free $U(\mathfrak{T})$-modules of rank $2$ over the  Neveu-Schwarz type $\S^{\frac{1}{2}}$.

\section{Non-weight $\S^\epsilon$-modules $\Omega_{\S^{\epsilon}}(\lambda,\alpha,h)$}
In this section,     we define a class of non-weight modules $\Omega_{\S^{\epsilon}}(\lambda,\alpha,h)$ over $\S^\epsilon$.

Let $V=V_{\bar0}\oplus V_{\bar1}$ be  $\Z_2$-graded vector space.  Any element $v\in V_{\bar0}$ (resp. $v\in V_{\bar1}$)
is said to be even   (resp. odd). Define $|v|=0$ if
$v$ is even and $|v|=1$ if $v$ is odd. Elements in  $V_{\bar0}$ or  $V_{\bar1}$ are called homogeneous.
 Throughout this paper, all elements in superalgebras and modules are homogenous unless
specified.

Assume that $\mathcal{G}$ is a Lie superalgebra, a $\mathcal{G}$-module is a $\Z_2$-graded vector space $V$ together
with a bilinear map $\mathcal{G}\times V\rightarrow V$, denoted $(x,v)\mapsto xv$ such that
$$x(yv)-(-1)^{|x||y|}y(xv)=[x,y]v$$
and
$$\mathcal{G}_{\bar i} V_{\bar j}\subseteq  V_{\bar i+\bar j}$$
for all $x, y \in \mathcal{G}, v \in V $. Thus there is a parity-change functor $\Pi$ on the category of
$\mathcal{G}$-modules to itself.    We use $U(\mathcal{G})$ to denote
the universal enveloping algebra.
All modules for Lie superalgebras considered in this paper are $\Z_2$-graded and
all irreducible modules are non-trivial.

Fix any $\lambda\in\C^*,\alpha\in\C$ and $h(t)\in\C[t]$. For any $m\in\Z$,  define
$$h_m(t)=mh(t)-m(m-1)\alpha\frac{h(t)-h(\alpha)}{t-\alpha}.$$
Then
\begin{eqnarray}\label{211}
&&nh_n(t)-mh_m(t)+n(t-n\alpha)\frac{\partial}{\partial t}(h_m(t))-m(t-m\alpha)\frac{\partial}{\partial t}(h_n(t)) \nonumber
\\&=&(n-m)h_{m+n}(t)
\end{eqnarray}
for any $m,n\in\Z$.
 Let $\Omega_{\mathcal{W}}(\lambda,\alpha,h)=\C[t,s]$
as a vector space and  define the  action of $\mathcal{W}$ as follows:
\begin{eqnarray*}
&&L_m(f(t,s))=\lambda^m(s+h_m(t))f(t,s-m)-m\lambda^m(t-m\alpha)\frac{\partial}{\partial t}(f(t,s-m)),
\\&&
W_m(f(t,s))=\lambda^m(t-m\alpha)f(t,s-m),\ C_1(f(t,s))=C_2(f(t,s))=0,
\end{eqnarray*}
where $m,n\in\Z,\alpha\in\C,h(t)\in\C[t].$

Now we recall some know   results of  $\Omega_{\mathcal{W}}(\lambda,\alpha,h)$.
\begin{lemm}\label{le2.1}{\rm (see \cite{CG})}.
For $\lambda\in\C^*,\alpha\in\C,h(t)\in\C[t]$,
then
\begin{itemize}
\item[\rm(1)] $\Omega_{\mathcal{W}}(\lambda,\alpha,h)$ is a $\mathcal{W}$-module;
\item[\rm(2)] $\Omega_{\mathcal{W}}(\lambda,\alpha,h)$ is simple if and only if  $\alpha\neq0$.
\end{itemize}
\end{lemm}
\begin{theo}\label{th2.1}{\rm (see \cite{CG})}.
Any free $U(\C L_0\oplus \C W_0)$-module of rank $1$ over   $\mathcal{W}$ is isomorphic to some  $\Omega_{\mathcal{W}}(\lambda,\alpha,h)$  for $\lambda\in\C^*,\alpha\in\C,h(t)\in\C[t].$
\end{theo}
Let $V=\C[t^2,s]\oplus t\C[t^2,s]$. Then $V$ is a $\Z_2$-graded vector space with $V_{\bar 0}=\C[t^2,s]$ and
$V_{\bar 1}=t\C[t^2,s]$.  A precise construction of an $\S^0$-module structure on
$V$ will be presented.

\begin{prop}\label{pro2.21}
For $\lambda\in \C^*,\alpha\in\C,h(t)\in\C[t],f(t^2,s)\in\C[t^2,s]$ and $tf(t^2,s)\in t\C[t^2,s]$,
we define the action
of Ramond type super-BMS$_3$ algebra $\S^{0}$ on $V$  as follows
\begin{eqnarray}\label{2.22}
L_mf(t^2,s)&=&\lambda^m(s+h_m(t^2))f(t^2,s-m) \nonumber
\\&&-m\lambda^m(t^2-m\alpha)\frac{\partial}{\partial t^2}(f(t^2,s-m)),
\\L_{m}tf(t^2,s)&=&\lambda^m(s-\frac{m}{2}+h_m(t^2))tf(t^2,s-m) \nonumber
\\&& \label{2.33} -m\lambda^m(t^2-m\alpha)t\frac{\partial}{\partial t^2}f(t^2,s-m),
\\ \label{2.44} W_mf(t^2,s)&=&\lambda^m(t^2-m\alpha)f(t^2,s-m),
\\ \label{2.55}  W_{m}tf(t^2,s)&=&\lambda^m(t^2-m\alpha)tf(t^2,s-m),
\\ \label{2.66} G_mf(t^2,s)&=&\lambda^mtf(t^2,s-m),
\\ \label{2.77} G_mtf(t^2,s)&=&\lambda^m(t^2-2m\alpha)f(t^2,s-m),
\\ \label{2.88} C_i(f(t^2,s))&=&C_i(tf(t^2,s))=0
\end{eqnarray}
for $m\in\Z,i=1,2.$
 Then $V$ is an $\S^{0}$-module under the action of \eqref{2.22}-\eqref{2.88}, which
  is
free of rank $1$ as a module over $U(\mathfrak{t})$ and   denoted by $\Omega_{\S^{0}}(\lambda,\alpha,h)$.
\end{prop}
\begin{proof}
By the trivial action of \eqref{2.88}, we omit the central elements in the following proof. From Lemma \ref{le2.1} (1),  we obtain
\begin{eqnarray*}
&&[L_m, L_n]f(t^2,s)=L_mL_nf(t^2,s)-L_nL_mf(t^2,s),
\\&&[L_m, W_n]f(t^2,s)=L_mW_nf(t^2,s)-W_nL_mf(t^2,s),
\\&&0=(W_mW_n-W_nW_m)f(t^2,s),
\end{eqnarray*}
where $m,n \in\Z$.
For   $m,n\in\Z$, it follows from  \eqref{2.33} and  \eqref{2.55}  that  we have
\begin{eqnarray}\label{2.99}
L_mL_ntf(t^2,s)&=&L_m\Big(\lambda^n(s-\frac{n}{2}+h_n(t^2))tf(t^2,s-n)\nonumber
\\&&-n\lambda^n(t^2-n\alpha)t\frac{\partial}{\partial t^2}f(t^2,s-n)\Big)\nonumber
%\\&=&\lambda^{m+n}\Big(s-\frac{m}{2}+h_m(t^2)\Big)t\Big(s-m-\frac{n}{2}+h_n(t^2)\Big)f(t^2,s-m-n)
%\\&&-m\lambda^{m+n}(t^2-m\alpha)t\frac{\partial}{\partial t^2}\Big(\big(s-m-\frac{n}{2}+h_n(t^2)\big)f(t^2,s-m-n)\Big)
%\\&&-n\lambda^{m+n}\Big(s-\frac{m}{2}+h_m(t^2)\Big)t(t^2-n\alpha)\frac{\partial}{\partial t^2}f(t^2,s-m-n)
%\\&&+mn\lambda^{m+n}(t^2-m\alpha)t\frac{\partial}{\partial t^2}\Big((t^2-n\alpha)t\frac{\partial}{\partial t^2}f(t^2,s-m-n)\Big)
\\&=&\lambda^{m+n}\Big(s-\frac{m}{2}+h_m(t^2)\Big)t\Big(s-m-\frac{n}{2}+h_n(t^2)\Big)f(t^2,s-m-n)\nonumber
\\&&-m\lambda^{m+n}(t^2-m\alpha)tf(t^2,s-m-n)\frac{\partial}{\partial t^2}h_n(t^2)\nonumber
\\&&-m\lambda^{m+n}(t^2-m\alpha)t\big(s-m-\frac{n}{2}+h_n(t^2)\big)\frac{\partial}{\partial t^2}f(t^2,s-m-n)\nonumber
\end{eqnarray}
\begin{eqnarray}
&&-n\lambda^{m+n}\Big(s-\frac{m}{2}+h_m(t^2)\Big)t(t^2-n\alpha)\frac{\partial}{\partial t^2}f(t^2,s-m-n)\nonumber
\\&&+mn\lambda^{m+n}(t^2-m\alpha)t\frac{\partial}{\partial t^2}f(t^2,s-m-n)\nonumber
\\&&+mn\lambda^{m+n}(t^2-m\alpha)t(t^2-n\alpha)\frac{\partial^2}{\partial^2 (t^2)}f(t^2,s-m-n),
\\L_mW_ntf(t^2,s)&=&L_m\Big(\lambda^n(t^2-n\alpha)tf(t^2,s-n)\Big)\nonumber
\\&=&\lambda^{m+n}(s-\frac{m}{2}+h_m(t^2))t(t^2-n\alpha)f(t^2,s-m-n) \nonumber
\\&&-m\lambda^{m+n}(t^2-m\alpha)tf(t^2,s-m-n)\nonumber
\\&& \label{299} -m\lambda^{m+n}(t^2-m\alpha)t(t^2-n\alpha)\frac{\partial}{\partial t^2}f(t^2,s-m-n)
\end{eqnarray}
and
\begin{eqnarray}
W_nL_mtf(t^2,s)&=&W_n\Big(\lambda^m(s-\frac{m}{2}+h_m(t^2))tf(t^2,s-m)\nonumber
\\&&-m\lambda^m(t^2-m\alpha)t\frac{\partial}{\partial t^2}f(t^2,s-m)\Big)\nonumber
\\&=&\lambda^{m+n}(t^2-n\alpha)(s-n-\frac{m}{2}+h_m(t^2))tf(t^2,s-m-n)\nonumber
\\&& \label{29900} -m\lambda^{m+n}(t^2-n\alpha)(t^2-m\alpha)t\frac{\partial}{\partial t^2}f(t^2,s-m-n).
\end{eqnarray}
Then for  $m,n\in\Z$, by   \eqref{211} and \eqref{2.99},   we compute that
\begin{eqnarray*}
&&[L_m,L_n]tf(t^2,s)
\\&=&\lambda^{m+n}\Big(s-\frac{m}{2}+h_m(t^2)\Big)t\Big(s-m-\frac{n}{2}+h_n(t^2)\Big)f(t^2,s-m-n)
\\&&-m\lambda^{m+n}(t^2-m\alpha)tf(t^2,s-m-n)\frac{\partial}{\partial t^2}h_n(t^2)
\\&&-m\lambda^{m+n}(t^2-m\alpha)t\big(s-m-\frac{n}{2}+h_n(t^2)\big)\frac{\partial}{\partial t^2}f(t^2,s-m-n)
\\&&-n\lambda^{m+n}\Big(s-\frac{m}{2}+h_m(t^2)\Big)t(t^2-n\alpha)\frac{\partial}{\partial t^2}f(t^2,s-m-n)
\\&&+mn\lambda^{m+n}(t^2-m\alpha)t\frac{\partial}{\partial t^2}f(t^2,s-m-n)
\\&&+mn\lambda^{m+n}(t^2-m\alpha)t(t^2-n\alpha)\frac{\partial^2}{\partial^2 (t^2)}f(t^2,s-m-n)
\end{eqnarray*}
\begin{eqnarray*}
&&-\lambda^{m+n}\Big(s-\frac{n}{2}+h_n(t^2)\Big)t\Big(s-n-\frac{m}{2}+h_m(t^2)\Big)f(t^2,s-m-n)
\\&&+n\lambda^{m+n}(t^2-n\alpha)tf(t^2,s-m-n)\frac{\partial}{\partial t^2}h_m(t^2)
\\&&+n\lambda^{m+n}(t^2-n\alpha)t\big(s-n-\frac{m}{2}+h_m(t^2)\big)\frac{\partial}{\partial t^2}f(t^2,s-m-n)
\\&&+m\lambda^{m+n}\Big(s-\frac{n}{2}+h_n(t^2)\Big)t(t^2-m\alpha)\frac{\partial}{\partial t^2}f(t^2,s-m-n)
\\&&-mn\lambda^{m+n}(t^2-n\alpha)t\frac{\partial}{\partial t^2}f(t^2,s-m-n)
\\&&-mn\lambda^{m+n}(t^2-n\alpha)t(t^2-m\alpha)\frac{\partial^2}{\partial^2 (t^2)}f(t^2,s-m-n),
\\&=&\lambda^{m+n}tf(t^2,s-m-n)\Big(n(s-\frac{n}{2})-m(s-\frac{m}{2})+nh_n(t^2)-mh_m(t^2)\Big)
\\&&+\lambda^{m+n}tf(t^2,s-m-n)\Big(n(t^2-n\alpha)\frac{\partial}{\partial t^2}h_m(t^2)-m(t^2-m\alpha)\frac{\partial}{\partial t^2}h_n(t^2)\Big)
\\&&+m^2\lambda^{m+n}(t^2-m\alpha)t\frac{\partial}{\partial t^2}f(t^2,s-m-n)
\\&&-n^2\lambda^{m+n}(t^2-n\alpha)t\frac{\partial}{\partial t^2}f(t^2,s-m-n)
\\&&+mn\lambda^{m+n}(n-m)\alpha t\frac{\partial}{\partial t^2}f(t^2,s-m-n)
\\&=&(n-m)\lambda^{m+n}(s-\frac{m+n}{2}+h_{m+n}(t^2))tf(t^2,s-m-n)
\\&&-(n-m)(m+n)\lambda^{m+n}(t^2-(m+n)\alpha)t\frac{\partial}{\partial t^2}f(t^2,s-m-n)
\\&=&(n-m)L_{m+n}tf(t^2,s).
\end{eqnarray*}
For   $m,n\in\Z$,  from \eqref{299} and \eqref{29900},   one can check that
\begin{eqnarray*}
[L_m,W_n]tf(t^2,s)&=&(n-m)\lambda^{m+n}(t^2-(m+n)\alpha) tf(t^2,s-m-n)
\\&=&(n-m)W_{m+n}tf(t^2,s).
\end{eqnarray*}
For   $m,n\in\Z$, by the similar computation, we conclude that
\begin{eqnarray*}
&&[L_m,G_n]f(t^2,s)
\\&=&(n-\frac{m}{2})\lambda^{m+n} tf(t^2,s-m-n)
\\&=&(n-\frac{m}{2})G_{m+n}f(t^2,s),
\end{eqnarray*}
\begin{eqnarray*}
&&[L_m,G_n]tf(t^2,s)
\\&=&(n-\frac{m}{2})\lambda^{m+r}(t^2-2(m+n)\alpha) f(t^2,s-m-n)
\\&=&(n-\frac{m}{2})G_{m+n}tf(t^2,s),
\\&&[G_m,G_n]f(t^2,s)
\\&=&\lambda^{m+n}(t^2-2m\alpha) tf(t^2,s-m-n)+\lambda^{m+n}(t^2-2n\alpha) tf(t^2,s-m-n)
\\&=&2W_{m+n}f(t^2,s),
\\&&
[G_m,G_n]tf(t^2,s)
\\&=&\lambda^{m+n}(t^2-2n\alpha) tf(t^2,s-m-n)+\lambda^{m+n}(t^2-2m\alpha) tf(t^2,s-m-n)
\\&=&2W_{m+n}tf(t^2,s).
\end{eqnarray*}
Finally, for   $m,n\in\Z$, it is easy to get that
\begin{eqnarray*}
[W_m,G_n]f(t^2,s)=[W_m,G_n]tf(t^2,s)=0.
\end{eqnarray*}
 This completes the proof.
\end{proof}

\begin{theo}\label{th2.3}
Let $\lambda\in\C^*,\alpha\in\C,h(t)\in\C[t]$.
Then   $\Omega_{\S^{0}}(\lambda,\alpha,h)$ is simple if and only if  $\alpha\neq0$.
\end{theo}
\begin{proof}
To prove this, we   consider two cases  in the following.
\begin{case}
$\alpha=0.$
\end{case}
For any $i\in\Z_+$, denote
$\pi_i:=(t^{2})^i\C[t^2,s]\oplus t\C[t^2,s]$. It is clear that for any $i\in\Z_+$, $\pi_i$ is a submodule
of $\Omega_{\S^{0}}(\lambda,0,h)$ for $\lambda\in\C^*$ and $h(t)\in\C[t]$.

\begin{case}
$\alpha\neq0.$
\end{case}
Let $P=P_{\bar0}\oplus P_{\bar 1}$ be a nonzero submodule of
$\Omega_{\S^{0}}(\lambda,\alpha,h)$.
It follows from    Lemma \ref{le2.1} (2)  that   $P_{\bar0}=\Omega_{\S^{0}}(\lambda,\alpha,h)_{\bar0}$
is a simple  $\S_{\bar0}$-module.
Taking any $f(t^2,s)\in P_{\bar0}=\C[t^2,s]$, we have $G_0f(t^2,s)=tf(t^2,s)\in P_{\bar1}$
 by \eqref{2.66}.
  Consequently, $P=\Omega_{\S^{0}}(\lambda,\alpha,h)$, which implies that $\Omega_{\S^{0}}(\lambda,\alpha,h)$
  is a
simple  $\S^{0}$-module.
\end{proof}
\begin{rema}
For $i\in\Z_+$, the quotient module
$$\bar \pi_i:=(t^{2})^i\C[t^2,s]\oplus t\C[t^2,s]/(t^{2})^{i+1}\C[t^2,s]\oplus t\C[t^2,s]$$
  is simple if and only if  $i\neq h(0)$.
  In fact, for any nonzero  $(t^2)^ig(s) \in\bar \pi_i$,
   we have
  \begin{eqnarray*}
L_m((t^2)^ig(s))&\equiv&\lambda^m(s+mh(0))(t^2)^ig(s-m)-\lambda^mmi(t^2)^ig(s-m)
\\&\equiv&\lambda^m(s+m(h(0)-i))(t^2)^ig(s-m)
 \quad \mathrm{mod}\  (t^{2})^{i+1}\C[t^2,s]\oplus t\C[t^2,s]
\end{eqnarray*}
 and
 $$W_m((t^2)^ig(s))=0.$$
\end{rema}
\begin{theo}\label{th2.4}
 Let $\lambda,\mu\in\C^*,\alpha,\beta\in\C,h(t),g(t)\in\C[t]$.
Then
\begin{itemize}
\item[\rm(i)]    $\Pi\big(\Omega_{\S^{0}}(\lambda,\alpha,h)\big)\ncong\Omega_{\S^{0}}(\lambda^\prime,\alpha^\prime,h^\prime)$ for any $\lambda^\prime,\alpha^\prime\in\C,h^\prime(t)\in\C[t]$;

\item[\rm(ii)]     $\Omega_{\S^{0}}(\lambda,\alpha,h)\cong\Omega_{\S^{0}}(\mu,\beta,g)$ if and only if $\lambda=\mu,\alpha=\beta$ and $h=g$.
\end{itemize}
\end{theo}
\begin{proof}
(i) Let
$$
\psi:\Pi\big(\Omega_{\S^{0}}(\lambda,\alpha,h)\big)\longrightarrow\Omega_{\S^{0}}(\lambda^\prime,\alpha^\prime,h^\prime)
$$ be an isomorphism. Let  $\mathbf{1}$ and  $\mathbf{1}^\prime$  be the
generators of the free $U(\mathfrak{t})$-modules
$\Pi\big(\Omega_{\S^{0}}(\lambda,\alpha,h)\big)$ and $\Omega_{\S^{0}}(\lambda^\prime,\alpha^\prime,h^\prime),$   respectively.
We see that there exists some $c,d\in\C^*$ such that $\psi(\mathbf{1})=c t\mathbf{1}^\prime$
and $\psi(t\mathbf{1})=d\mathbf{1}^\prime$.
For $r\in\Z$, we have $$G_r\psi(\mathbf{1})=G_r\big(c t\mathbf{1}^\prime\big)
=c\lambda^r(t^2-2r\alpha^\prime)\mathbf{1}^\prime
\ \mathrm{and} \  \psi(G_r\mathbf{1})=d\lambda^r\mathbf{1}^\prime.$$  Clearly, $G_r\psi(\mathbf{1})\neq\psi(G_r\mathbf{1})$,  which yields a contradiction.

(ii)
The sufficiency is obvious. Suppose $\Omega_{\S^{0}}(\lambda,\alpha,h)\cong\Omega_{\S^{0}}(\mu,\beta,g)$ as $\S^{0}$-module. Then
 $\Omega_{\S^{0}}(\lambda,\alpha,h)_{\bar 0}\cong\Omega_{\S^{0}}(\mu,\beta,g)_{\bar 0}$ as $\S_{\bar 0}$-modules.  By Theorem 3.2 in \cite{CG}, we have  $\lambda=\mu,\alpha=\beta$ and $h=g$.
\end{proof}

Assume that  $M=\C[t,s]\oplus\C[y,x]$. Then $M$ is a $\Z_2$-graded vector space with $M_{\bar 0}=\C[t,s]$ and $M_{\bar 1}=\C[y,x]$.
The  $\S^{\frac{1}{2}}$-module structure on $M$    will be characterized as follows.
\begin{prop}
For $\lambda\in \C^*,\alpha\in\C,h(t)\in\C[t],h(y)\in\C[y],f(t,s)\in\C[t,s]$ and $k(y,x)\in\C[y,x]$,
we define the action
of Neveu-Schwarz type  super-BMS$_3$ algebra $\S^{\frac{1}{2}}$ on $M$  as follows
\begin{eqnarray}\label{3.22}
L_mf(t,s)&=&\lambda^m(s+h_m(t))f(t,s-m) \nonumber
\\&&-m\lambda^m(t-m\alpha)\frac{\partial}{\partial t}(f(t,s-m)),
\\L_{m}k(y,x)&=&\lambda^m(x-\frac{m}{2}+h_m(y))k(y,x-m) \nonumber
\\&&-m\lambda^m(y-m\alpha)\frac{\partial}{\partial y}k(y,x-m),
\\W_mf(t,s)&=&\lambda^m(t-m\alpha)f(t,s-m),
\\W_{m}k(y,x)&=&\lambda^m(y-m\alpha)k(y,x-m),
\\G_rf(t,s)&=&\lambda^{r-\frac{1}{2}}f(y,x-r),
\\G_rk(y,x)&=&\lambda^{r+\frac{1}{2}}(t-2r\alpha)k(t,s-r),
\\ \label{3.88} C_i(f(t,s))&=&C_i(k(y,x))=0,
\end{eqnarray}
where $m\in\Z,i=1,2,r\in\frac{1}{2}+\Z$.
 Then $M$ is an $\S^{\frac{1}{2}}$-module under the action of \eqref{3.22}-\eqref{3.88}, which
  is a
free of rank $2$ as a module over $U(\mathfrak{T})$ and   denoted by  $\Omega_{\S^{\frac{1}{2}}}(\lambda,\alpha,h)$.
\end{prop}
\begin{proof}
By the similar calculations appeared in Lemma \ref{le2.1} (1) and Proposition \ref{pro2.21},  we have
\begin{eqnarray*}
&&[L_m, L_n]f(t,s)=L_mL_nf(t,s)-L_nL_mf(t,s),
\\&&[L_m, L_n]k(y,x)=L_mL_nk(y,x)-L_nL_mk(y,x),
\\&&[L_m, W_n]f(t,s)=L_mW_nf(t,s)-W_nL_mf(t,s),
\\&&[L_m, W_n]k(y,x)=L_mW_nk(y,x)-W_nL_mk(y,x),
\\&&(W_mW_n-W_nW_m)f(t,s)=(W_mW_n-W_nW_m)k(y,x)=0,
\end{eqnarray*}
where $m,n \in\Z$.
For any $m\in\Z,r\in\frac{1}{2}+\Z$, we deduce that
\begin{eqnarray*}
&&[L_m,G_r]f(t,s)
\\&=&L_m\Big(\lambda^{r-\frac{1}{2}}f(y,x-r)\Big)
\\&&-G_r\Big(\lambda^m(s+h_m(t))f(t,s-m) \nonumber
-m\lambda^m(t-m\alpha)\frac{\partial}{\partial t}(f(t,s-m))\Big)
\\&=&(r-\frac{m}{2})\lambda^{m+r-\frac{1}{2}} f(y,x-m-r)
\\&=&(r-\frac{m}{2})G_{m+r}f(t,s).
\end{eqnarray*}
For any $m\in\Z,r,q\in\frac{1}{2}+\Z$,   we obtain that
\begin{eqnarray*}
&&[L_m,G_r]k(y,x)
\\&=&(r-\frac{m}{2})\lambda^{m+r+\frac{1}{2}}(t-2(m+r)\alpha) f(t,s-m-r)
\\&=&(r-\frac{m}{2})G_{m+r}k(y,x),
\\&&
[G_r,G_p]f(t,s)
\\&=&2\lambda^{r+p}(t-(r+p)\alpha)f(t,s-r-p)
\\&=&2W_{r+p}f(t,s),
\\&&
[G_r,G_p]k(y,x)
\\&=&2\lambda^{r+p}(y-(r+p)\alpha)k(y,x-r-p)
\\&=&2W_{r+p}k(y,x).
\end{eqnarray*}
At last, for any $m\in\Z,r\in\frac{1}{2}+\Z$, it is easy to get
\begin{eqnarray*}
[W_m,G_r]f(t,s)=[W_m,G_r]k(y,x)=0.
\end{eqnarray*}
This proposition holds.
\end{proof}

Assume that $\sigma: \S^{\frac{1}{2}}\rightarrow\S^{0}$ is a linear map defined by
 \begin{eqnarray*}
L_m&\mapsto&\frac{1}{2}L_{2m}+\frac{1}{16}\delta_{m,0}C_1,
\\W_m&\mapsto&\frac{1}{2}W_{2m}+\frac{1}{16}\delta_{m,0}C_2,
\\G_r&\mapsto&\frac{1}{\sqrt{2}}G_{2r},
\\C_1&\mapsto& 2C_1,
\\C_2&\mapsto& 2C_2
\end{eqnarray*}
for $m\in\Z, r\in\frac{1}{2}+\Z$. It is easy to show that $\sigma$ is an injective Lie superalgebra
homomorphism. Clearly,  $\S^{\frac{1}{2}}$ can be regarded as a subalgebra of $\S^{0}$. Then $\Omega_{\S^{0}}(\lambda,\alpha,h)$
  can be regarded as  an $\S^{\frac{1}{2}}$-module.

\begin{prop}\label{prop2.6}
$\Omega_{\S^{0}}(\lambda,\alpha,h)$ is a simple  $\S^{\frac{1}{2}}$-module if and only if it is a simple $\S^{0}$-module.
\end{prop}
\begin{proof}
We only prove sufficiency. From
\begin{eqnarray*}
&&\sigma(L_{m})f(t^2,s)=\frac{1}{2}\lambda^{2m}(s+h_{2m}(t^2))f(t^2,s-2m)
-m\lambda^{2m}(t^2-2m\alpha)\frac{\partial}{\partial t^2}(f(t^2,s-2m)),
\\&& \sigma(W_{m})f(t^2,s)=\frac{1}{2}\lambda^{2m}(t^2-2m\alpha)f(t^2,s-2m)
\end{eqnarray*}
and $\alpha\neq0$, we see that $\Omega_{\S^{0}}(\lambda,\alpha,h)_{\bar 0}$ is a simple  $\sigma(\mathcal{W})$-module.
Then by $\sigma(G_r)f(t^2,s)=\frac{1}{\sqrt{2}}\lambda^{2r}tf(t^2,s-2r)$, one has $t\mathbf{1}\in\Omega_{\S^{0}}(\lambda,\alpha,h)_{\bar 1}$, which     yields that $\Omega_{\S^{0}}(\lambda,\alpha,h)$ is a simple $\S^{\frac{1}{2}}$-module.
\end{proof}

\begin{prop}\label{prop2.7}
Let  $\lambda\in\C^*,\alpha\in\C$ and $g(t),h(t)\in\C[t]$.  Suppose that for any $m\in\Z$  $g_m(\frac{1}{2}t)=\frac{1}{2}h_{2m}(t)$. Then as $\S^{\frac{1}{2}}$-modules,
we have  $\Omega_{\S^{\frac{1}{2}}}(\lambda,\alpha,g)\cong\Omega_{\S^{0}}(\sqrt{\lambda},\alpha,h)$.
\end{prop}
\begin{proof}
Let
\begin{eqnarray*}
\Psi:\Omega_{\S^{\frac{1}{2}}}(\lambda,\alpha,g)&\longrightarrow&\Omega_{\S^{0}}(\sqrt{\lambda},\alpha,h)
\\f(t,s)&\longmapsto&f(\frac{1}{2}t^2,\frac{1}{2}s)
\\k(y,x)&\longmapsto&\sqrt{\frac{\lambda}{2}}tk(\frac{1}{2}t^2,\frac{1}{2}s)
\end{eqnarray*}
be the linear map from  $\Omega_{\S^{\frac{1}{2}}}(\lambda,\alpha,g)$ to $\Omega_{\S^{0}}(\sqrt{\lambda},\alpha,h)$.
 Clearly, $\Psi$ is bijective.  Now we   show that it is an $\S^{\frac{1}{2}}$-module isomorphism.

For any $m\in\Z, r\in\frac{1}{2}+\Z, g(t),h(t)\in\C[t],f(t,s)\in\C[t,s], k(y,x)\in\C[y,x]$, we check that
\begin{eqnarray*}
&&\Psi(L_mf(t,s))
\\&=&\Psi\Big(\lambda^m(s+g_m(t))f(t,s-m)
\\&&-m\lambda^m(t-m\alpha)\frac{\partial}{\partial t}(f(t,s-m))\Big)
\\&=&\lambda^m(\frac{s}{2}+g_m(\frac{1}{2}t^2))f(\frac{1}{2}t^2,\frac{1}{2}s-m)
\\&&-m\lambda^m(\frac{1}{2}t^2-m\alpha)\frac{\partial}{\partial (\frac{1}{2}t^2)}(f(\frac{1}{2}t^2,\frac{1}{2}s-m)),
\\&& \Psi(L_{m}k(y,x))
\\&=&\Psi\Big(\lambda^m(x-\frac{m}{2}+g_m(y))k(y,x-m)
\\&&-m\lambda^m(y-m\alpha)\frac{\partial}{\partial y}k(y,x-m)\Big)
\\&=&\sqrt{\frac{\lambda}{2}}\lambda^m(\frac{1}{2}s-\frac{m}{2}+g_m(\frac{1}{2}t^2))tk(\frac{1}{2}t^2,\frac{1}{2}s-m)
\\&&-m\sqrt{\frac{\lambda}{2}}\lambda^m(\frac{1}{2}t^2-m\alpha)t\frac{\partial}{\partial (\frac{1}{2}t^2)}k(\frac{1}{2}t^2,\frac{1}{2}s-m),
\\&& \Psi(W_mf(t,s))
\\&=&\Psi\Big(\lambda^m(t-m\alpha)f(t,s-m)\Big)=\lambda^m(\frac{1}{2}t^2-m\alpha)f(\frac{1}{2}t^2,\frac{1}{2}s-m),
\\&& \Psi(W_{m}k(y,x))
\\&=&\lambda^m(y-m\alpha)k(y,x-m)\Big)=\sqrt{\frac{\lambda}{2}}\lambda^m(\frac{1}{2}t^2-m\alpha)tk(\frac{1}{2}t^2,\frac{1}{2}s-m),
\\&& \Psi(G_rf(t,s))
\\&=&\Psi\Big(\lambda^{r-\frac{1}{2}}f(y,x-r)\Big)
=\frac{\lambda^{r}}{\sqrt{2}}tf(\frac{1}{2}t^2,\frac{1}{2}s-r),
\\&& \Psi(G_rk(y,x))
\\&=&\Psi\Big(\lambda^{r+\frac{1}{2}}(t-2r\alpha)k(t,s-r)\Big)
=\lambda^{r+\frac{1}{2}}(\frac{1}{2}t^2-2r\alpha)k(\frac{1}{2}t^2,\frac{1}{2}s-r),
\\&&  \Psi(C_i(f(t,s)))=\Psi(C_i(k(y,x)))=0
\end{eqnarray*}
for $i=1,2.$
Then by the following calculation
\begin{eqnarray*}
&&\sigma(L_m)\Psi(f(t,s))=\frac{1}{2}L_{2m}f(\frac{1}{2}t^2,\frac{1}{2}s)
\\&=&
\frac{1}{2} \lambda^{m}(s+h_{2m}(t^2))f(\frac{1}{2}t^2,\frac{1}{2}(s-2m))
\\&&-m\lambda^m(t^2-2m\alpha)\frac{\partial}{\partial (t^2)}(f(\frac{1}{2}t^2,\frac{1}{2}(s-2m))),
\\&& \sigma(L_{m})\Psi(k(y,x))=\frac{1}{2}L_{2m}\sqrt{\frac{\lambda}{2}}tk(\frac{1}{2}t^2,\frac{1}{2}s)
\\&=&
\sqrt{\frac{\lambda}{2}}\lambda^m(\frac{s}{2}-\frac{m}{2}+\frac{1}{2}h_{2m}(t^2))tk(\frac{1}{2}t^2,\frac{1}{2}s-m)
\\&&-m\sqrt{\frac{\lambda}{2}}\lambda^m(\frac{1}{2}t^2-m\alpha)t\frac{\partial}{\partial (\frac{1}{2}t^2)}k(\frac{1}{2}t^2,\frac{1}{2}s-m),
\\&& \sigma(W_m)\Psi(f(t,s))=\frac{1}{2}W_{2m}f(\frac{1}{2}t^2,\frac{1}{2}s)
\\&=&\lambda^m(\frac{1}{2}t^2-m\alpha)f(\frac{1}{2}t^2,\frac{s}{2}-m),
\\&& \sigma(W_{m})\Psi(k(y,x))=\frac{1}{2}W_{2m}\sqrt{\frac{\lambda}{2}}tk(\frac{1}{2}t^2,\frac{1}{2}s)
\\&=&\sqrt{\frac{\lambda}{2}}\lambda^m(\frac{1}{2}t^2-m\alpha)tk(\frac{1}{2}t^2,\frac{s}{2}-m),
\\&& \sigma(G_r)\Psi(f(t,s))=\frac{1}{\sqrt{2}}G_{2r}f(\frac{1}{2}t^2,\frac{1}{2}s)
\\&=&\frac{1}{\sqrt{2}}\lambda^rf(\frac{1}{2}t^2,\frac{1}{2}s-r),
\\&& \sigma(G_r)\Psi(k(y,x))=\frac{1}{\sqrt{2}}G_{2r}\sqrt{\frac{\lambda}{2}}tk(\frac{1}{2}t^2,\frac{1}{2}s)
\\&=&\lambda^{r+\frac{1}{2}}(\frac{1}{2}t^2-2r\alpha)k(\frac{1}{2}t^2,\frac{1}{2}s-r),
\\&&   \sigma(C_i(f(t,s)))=\sigma(C_i(k(y,x)))=0,
\end{eqnarray*}
one has
\begin{eqnarray*}
&&\Psi(L_mf(t,s))=\sigma(L_m)\Psi(f(t,s)), \Psi(L_{m}k(y,x))=\sigma(L_{m})\Psi(k(y,x)),
\\&&\Psi(W_mf(t,s))= \sigma(W_m)\Psi(f(t,s)),\Psi(W_{m}k(y,x))=\sigma(W_{m})\Psi(k(y,x)),
\\&&\Psi(G_rf(t,s))=\sigma(G_r)\Psi(f(t,s)),\Psi(G_rk(y,x))=\sigma(G_r)\Psi(k(y,x)),
\\&&\Psi(C_i(f(t,s)))=\sigma(C_i(f(t,s)))=\Psi(C_i(k(y,x)))=\sigma(C_i(k(y,x)))=0
\end{eqnarray*}
for $i=1,2$.
Thus, we conclude  that $\Psi$   is an $\S^{\frac{1}{2}}$-module isomorphism.
\end{proof}
Now we  give  an example of the conditions  in Proposition \ref{prop2.7}.
\begin{exam} Let  $h(t)=t,g(\frac{1}{2}t)=t-\alpha$. For   $m\in\Z$,  we obtain
$$
g_m(\frac{1}{2}t)=m(t+\alpha)-2m^2\alpha=\frac{1}{2}h_{2m}(t).
$$
\end{exam}

 According to   Propositions \ref{prop2.6},    \ref{prop2.7} and Theorem \ref{th2.3}, we have   the following corollary.
\begin{coro} Let $\lambda\in\C^*,\alpha\in\C,h(t)\in\C[t]$. Then
 $\Omega_{\S^{\frac{1}{2}}}(\lambda,\alpha,h)$
is simple if and only if   $\alpha\neq0$.
\end{coro}

 By the similar arguments   in the proof of Theorem \ref{th2.4}, it yields the  following theorem.
\begin{theo}
 Let $\lambda,\mu\in\C^*,\alpha,\beta\in\C,h(t),g(t)\in\C[t]$.
Then  \begin{itemize}
\item[\rm(i)]    $\Pi\big(\Omega_{\S^{\frac{1}{2}}}(\lambda,\alpha,h)\big)\ncong\Omega_{\S^{\frac{1}{2}}}(\lambda^\prime,\alpha^\prime,h^\prime)$ for any $\lambda^\prime,\alpha^\prime\in\C,h^\prime(t)\in\C[t]$;

\item[\rm(ii)]     $\Omega_{\S^{\frac{1}{2}}}(\lambda,\alpha,h)\cong\Omega_{\S^{\frac{1}{2}}}(\mu,\beta,g)$ if and only if $\lambda=\mu,\alpha=\beta$ and $h=g$.
\end{itemize}
\end{theo}

\section{ Classification of free U($\mathfrak{t}$)-modules of rank 1 over   $\S^{0}$}
Assume that  $V=V_{\bar 0}\oplus V_{\bar1}$ is an $\S^0$-module such that it is free of rank $1$ as a $U(\mathfrak{t})$-module,
where $\mathfrak{t}=\C L_0\oplus \C W_0\oplus \C G_0$. From  the definition of $\S^0$ in \eqref{def1.1},   we have
$$L_0W_0=W_0L_0,\ L_0G_0=G_0L_0,\ W_0G_0=G_0W_0  \ \mathrm{and}\ G_0^2=W_0.$$
 Thus $U(\mathfrak{t})=\C[L_0,W_0]\oplus G_0\C[L_0,W_0]$. Choose a homogeneous basis element $\mathbf{1}$ in $V$. Without loss of generality, up
to a parity-change, we may assume $\mathbf{1}\in V_{\bar 0}$  and
$$V=U(\mathfrak{t})\mathbf{1}=\C[L_0,W_0]\mathbf{1}\oplus G_0\C[L_0,W_0]\mathbf{1}$$
with $V_{\bar 0}=\C[L_0,W_0]\mathbf{1}$ and $V_{\bar1}=G_0\C[L_0,W_0]\mathbf{1}$.
We can  suppose that $L_0\mathbf{1}=s \mathbf{1},W_0\mathbf{1}=t^2 \mathbf{1},G_0\mathbf{1}=t\mathbf{1}$.
In the following section,  $V=V_{\bar 0}\oplus V_{\bar1}$ is equal to
$\C[t^2,s]\mathbf{1}\oplus t\C[t^2,s]\mathbf{1}$ with $V_{\bar0}=\C[t^2,s]\mathbf{1}$ and $V_{\bar1}=t\C[t^2,s]\mathbf{1}$.
\begin{remark} The free $U(\mathfrak{t})$-module $V$ of rank $1$ over  $\S^{0}$  can be viewed  as a free $U(\mathfrak{T})$-module of
rank $2$ by defining $\mathbf{1}_{\bar 0}:=\mathbf{1},\mathbf{1}_{\bar 1}:=G_0\mathbf{1}$.
\end{remark}

According to  $\S_{\bar 0}\cong\mathcal{W}$,
it is clear  that $V_{\bar 0}$  can be regarded as a $\mathcal{W}$-module which is free of rank $1$ as a   $U(\mathfrak{T})$-module.
From  Lemma \ref{le2.1} and  Theorem \ref{th2.1},  for
  $m\in\Z$, $f(t^2,s)\in \C[t^2,s]$,  there exist $\lambda\in\C^*,\alpha\in\C,h(t)\in\C[t]$ such that
\begin{eqnarray}\nonumber
&&L_m(f(t^2,s))=\lambda^m(s+h_m(t^2))f(t^2,s-m)-m\lambda^m(t^2-m\alpha)\frac{\partial}{\partial t^2}(f(t^2,s-m)),
\\&& \label{eq3.2}
W_m(f(t^2,s))=\lambda^m(t^2-m\alpha)f(t^2,s-m),\ C_1(f(t^2,s))=C_2(f(t^2,s))=0.
\end{eqnarray}

Firstly, we give the following results.
 \begin{lemm}\label{lem3.2}
For $m\in\Z, f(t^2,s)\in\C[t^2,s],tf(t^2,s)\in t\C[t^2,s]$, we obtain
\begin{itemize}
\item[\rm(i)] $G_mtf(t^2,s)\mathbf{1}=f(t^2,s-m)G_mt\mathbf{1}$;

\item[\rm(ii)] $G_mf(t^2,s)\mathbf{1}=f(t^2,s-m)G_m\mathbf{1}$;

\item[\rm(iii)] $C_itf(t^2,s)\mathbf{1}=0$ for $i=1,2.$
\end{itemize}
\end{lemm}
\begin{proof}  (i)  It is easy to get that
$G_mL_0t\mathbf{1}=(L_0-m)G_mt\mathbf{1}$ by  the
relations of $\mathfrak{S}^0$.
Recursively, we conclude  that
$G_mL_0^nt\mathbf{1}=(L_0-m)^nG_mt\mathbf{1}$ for $n\in\Z_+$.
Hence,
$G_mtf(t^2,s)\mathbf{1}=f(t^2,s-m)G_mt\mathbf{1}$ for $m\in\Z$.
Similarly, we obtain (ii).

(iii) By \eqref{eq3.2}, we know that $C_if(t^2,s)\mathbf{1}=0$, which gives $C_i\mathbf{1}=0$  for $i=1,2$.
Thus   $C_itf(t^2,s)\mathbf{1}=tf(t^2,s)C_i\mathbf{1}=0$ for $i=1,2$.
\end{proof}

 \begin{lemm}\label{le3.33}
For $m\in\Z$, we  get $G_m\mathbf{1}=\lambda^mt\mathbf{1}$.
\end{lemm}
\begin{proof} To prove this, we  suppose
\begin{eqnarray*}
G_m\mathbf{1} = tf_m(t^2,s)\mathbf{1}\in t\C[t^2,s]\mathbf{1}
\end{eqnarray*} for $m\in\Z$.
For any $m\in\Z$, by $G_m^2\mathbf{1}=W_{2m}\mathbf{1}$, we have
\begin{eqnarray}\label{le3.66}
f_m(t^2,s-m)(2\lambda^m(t^2-m\alpha)-t^2f_m(t^2,s))=\lambda^{2m}(t^2-2m\alpha).
\end{eqnarray}
Considering the degree of $s$ and $t^2$ in above equation, we conclude that
$f_m(t^2,s-m)=f_m\in\C$.
We rewrite \eqref{le3.66} as
$$f_m((2\lambda^m-f_m)t^2-2m\lambda^m\alpha)=\lambda^{2m}(t^2-2m\alpha),$$
which implies $f_m=\lambda^m$.
\end{proof}

Now we present the main result of this section, which gives a
complete classification of free $U(\mathfrak{t})$-modules of rank 1 over $\S^0$.
\begin{theo}\label{th3.5}
 Assume that $V$ is an  $\S^0$-module such that the restriction of $V$ as a $U(\mathfrak{t})$-module is free of rank 1.
 Up to a
parity-change, we have $V\cong\Omega_{\S^0}(\lambda,\alpha,h)$ for $\lambda\in\C^*,\alpha\in\C,h(t)\in\C[t]$ with the $\S^0$-module structure defined as
in \eqref{2.22}-\eqref{2.88}.
\end{theo}

\begin{proof}
%For  any $f(t^2,s) \in\C[t^2,s]$,  we check that $C_if(t^2,s)\mathbf{1}=C_itf(t^2,s)\mathbf{1}=0$ for $i=1,2$.
For any $f(t^2,s) \in\C[t^2,s]$, by Lemma \ref{lem3.2} (ii) and Lemma \ref{le3.33}, we obtain
\begin{eqnarray}
G_mf(t^2,s)\mathbf{1}=f(t^2,s-m)G_m\mathbf{1}=\lambda^mf(t^2,s-m)t\mathbf{1}.
\end{eqnarray}
Moreover, we have
\begin{eqnarray*}
&&L_{m}tf(t^2,s)\mathbf{1}
\\&=&L_{m}G_0f(t^2,s)\mathbf{1}
\\&=&G_0L_{m}f(t^2,s)\mathbf{1}-\frac{m}{2}G_{m}f(t^2,s)\mathbf{1}
\\&=&\lambda^m(s-\frac{m}{2}+h_m(t^2))f(t^2,s-m)t\mathbf{1}-m\lambda^m(t^2-m\alpha)\frac{\partial}{\partial t^2}(f(t^2,s-m))t\mathbf{1}
\end{eqnarray*}
and
\begin{eqnarray*}
W_{m}tf(t^2,s)\mathbf{1}=W_{m}G_0f(t^2,s)\mathbf{1}=G_0W_{m}f(t^2,s)\mathbf{1}=\lambda^m(t^2-m\alpha)f(t^2,s-m)t\mathbf{1}.
\end{eqnarray*}
Suppose $G_mt\mathbf{1}=g_m(t^2,s)\mathbf{1}\in\C[t^2,s]\mathbf{1}$ for $m\in\Z$. From $G_mG_mt\mathbf{1}=W_{2m}t\mathbf{1}$, it is easy to check that
$g_m(t^2,s)=\lambda^m(t^2-2m\alpha)$.
Thus  $G_mt\mathbf{1}=\lambda^m(t^2-2m\alpha)\mathbf{1}$ for $m\in\Z$.
Then by  Lemma \ref{lem3.2} (i), one has
$$G_mtf(t^2,s)\mathbf{1}=\lambda^m(t^2-2m\alpha)f(t^2,s-m)t\mathbf{1}.$$
\end{proof}

\section{ Classification of free U($\mathfrak{T}$)-modules of rank $2$ over   $\S^{\frac{1}{2}}$}
By the definition of     $\S^{\frac{1}{2}}$, we see that it has a $2$-dimensional Cartan subalgebra $\mathfrak{T}=
\C L_0\oplus \C W_0$ (modulo center). Then we have  $U(\mathfrak{T})=\C[L_0,W_0]$. At the same time, we see that $\{W_m\mid m\in\Z\}$ can be
generated by odd elements $G_{\frac{1}{2}+r}$ for $r\in \Z$. Considering  the  pure even or pure odd non-trivial $\S^{\frac{1}{2}}$-modules, we obtain that  the action of $G_{\frac{1}{2}+r}$ on those modules is trivial for $r\in \Z$. This  implies that the action of $W_{m}$ is also trivial for $m\in \Z$.
 Thus, the  pure even or pure odd non-trivial     free $U(\mathfrak{T})$-modules of rank $1$ over $\S^{\frac{1}{2}}$ are isomorphic to Virasoro modules.
 In the following section, we will classify the free $U(\mathfrak{T})$-modules of rank $2$ over
$\S^{\frac{1}{2}}$.

Assume that  $M=M_{\bar 0}\oplus M_{\bar1}$ is an $\S^{\frac{1}{2}}$-module
 such that it is free of rank $2$ as a $U(\mathfrak{T})$-module
with two homogeneous basis elements $v$ and $w$. If the parities of $v$ and $w$ are the same, for any $r\in\Z+\frac{1}{2}$ then
 $G_{\pm r}v=G_{\pm r}w=0$. Therefore,
$$W_{0}v=\frac{1}{2}[G_{r},G_{-r}]v=0,  W_{0}w=\frac{1}{2}[G_{r},G_{-r}]w=0,$$
which gives a contradiction.
Then we conclude that $v$ and $w$ have different parities. Set $v=\mathbf{1}_{\bar 0}\in M_{\bar 0}$ and
$w=\mathbf{1}_{\bar 1}\in M_{\bar 1}$. As a vector space, we get $M_{\bar 0}= \C[t,s]\mathbf{1}_{\bar 0}$ and $M_{\bar 1}= \C[y,x]\mathbf{1}_{\bar 1}$.

By  $\S_{\bar0}\cong\mathcal{W}$,  we know that
  $M_{\bar 0}$ and $M_{\bar 1}$  are both can be viewed     as $\mathcal{W}$-modules.
  According to Lemma \ref{le2.1} and Theorem \ref{th2.1},
there exist
 $\lambda,\mu\in \C^*,\alpha,\beta\in\C,h(t)\in\C[t]$ and $g(y)\in\C[y]$
such that
\begin{eqnarray}\label{eqq4.1}
L_mf(t,s)&=&\lambda^m(s+h_m(t))f(t,s-m) \nonumber
\\&&-m\lambda^m(t-m\alpha)\frac{\partial}{\partial t}(f(t,s-m)),
\\ \label{eqq4.2} W_mf(t,s)&=&\lambda^m(t-m\alpha)f(t,s-m), C_i(f(t,s))=0
\end{eqnarray}
for $i=1,2,f(t,s)\in\C[t,s]$  and
\begin{eqnarray}\label{eqq4.3}
L_mf(y,x)&=&\mu^m(x+g_m(y))f(y,x-m) \nonumber
\\&&-m\mu^m(y-m\beta)\frac{\partial}{\partial y}(f(y,x-m)),
\\ \label{eqq4.4} W_mf(y,x)&=&\mu^m(y-m\beta)f(y,x-m), C_i(f(y,x))=0
\end{eqnarray}
for $i=1,2,f(y,x)\in\C[y,x]$.

Now we give    two preliminary results for later
use.
 \begin{lemm}\label{lem4.21452}
 Let $\lambda,\mu\in\C^*,\alpha,\beta\in\C,f(t,s)\in\C[t,s],k(y,x)\in\C[y,x]$. Then $\lambda=\mu,\alpha=\beta$ and there exists $c\in\C^*$ such that one
of the following two cases occurs.
\begin{itemize}
\item[\rm(i)] $G_{\frac{1}{2}}\mathbf{1}_{\bar0}=
    c\mathbf{1}_{\bar1},G_{\frac{1}{2}}\mathbf{1}_{\bar1}
    =\frac{1}{c}\lambda(t-\alpha)\mathbf{1}_{\bar0}$;

\item[\rm(ii)] $G_{\frac{1}{2}}\mathbf{1}_{\bar0}=\frac{1}{c}\lambda(y-\alpha)\mathbf{1}_{\bar1},G_{\frac{1}{2}}1_{\bar1}=c\mathbf{1}_{\bar0}$.
\end{itemize}
\end{lemm}
\begin{proof} Assume that  $G_{\frac{1}{2}}\mathbf{1}_{\bar0}=f(y,x)\mathbf{1}_{\bar1}$ and $G_{\frac{1}{2}}\mathbf{1}_{\bar1}=g(t,s)\mathbf{1}_{\bar0}$.
According to  $[G_{\frac{1}{2}},G_{\frac{1}{2}}]\mathbf{1}_{\bar0}=2W_1\mathbf{1}_{\bar0}$, we get
$$G_{\frac{1}{2}}^2\mathbf{1}_{\bar0}=f(W_0,L_0-\frac{1}{2})G_{\frac{1}{2}}\mathbf{1}_{\bar1}=f(t,s-\frac{1}{2})g(t,s)\mathbf{1}_{\bar0}$$
and $G_{\frac{1}{2}}^2\mathbf{1}_{\bar0}=W_1\mathbf{1}_{\bar0}=\lambda(t-\alpha)$, which imply $f(t,s-\frac{1}{2})g(t,s)=\lambda(t-\alpha)$.
Therefore,
\begin{eqnarray}\label{eq4.1}
f(t,s-\frac{1}{2})=c,g(t,s) =\frac{1}{c}\lambda(t-\alpha)\quad \mathrm{or} \quad f(t,s-\frac{1}{2})=\frac{1}{c}\lambda(t-\alpha),g(t,s) =c
\end{eqnarray} for $c\in\C^*$.
Similarly, by $G_{\frac{1}{2}}^2\mathbf{1}_{\bar1}=W_1\mathbf{1}_{\bar1}$, we see that  $g(y,x-\frac{1}{2})f(y,x)=\mu(y-\beta)$.
Then by \eqref{eq4.1},
we always obtain  $g(y,x-\frac{1}{2})f(y,x)=\lambda(y-\alpha)$,    which gives     $\lambda=\mu,\alpha=\beta$.
\end{proof}

Due to Lemma  \ref{lem4.21452}, up to a parity-change, we can assume   $\lambda=\mu,\alpha=\beta, G_{\frac{1}{2}}\mathbf{1}_{\bar0}= \mathbf{1}_{\bar1}$ and $G_{\frac{1}{2}}\mathbf{1}_{\bar1}=\lambda(t-\alpha)\mathbf{1}_{\bar0}$
  without loss of generality.
   \begin{lemm}\label{lem43334.33}
For any $m\in\Z,r \in \frac{1}{2}+\Z,$ we get
\begin{itemize}
\item[\rm(i)] $g_m(y)-h_m(y)=-\frac{m}{2};$
\item[\rm(ii)]
$G_r\mathbf{1}_{\bar 0}=\lambda^{r-\frac{1}{2}}\mathbf{1}_{\bar 1};$
\item[\rm(iii)]
$G_r\mathbf{1}_{\bar 1}=\lambda^{r+\frac{1}{2}}(t-2r\alpha)\mathbf{1}_{\bar 0}.$
\end{itemize}
\end{lemm}
\begin{proof}
 (i)
Fix $r\in  \frac{1}{2}+\Z$ and $r\neq\frac{3}{2}$. Based on  Lemma  \ref{lem4.21452}, we have
\begin{eqnarray}\label{4.635}
&& (\frac{3}{4}-\frac{1}{2}r)G_r\mathbf{1}_{\bar0}
=[L_{r-\frac{1}{2}},G_\frac{1}{2}]\mathbf{1}_{\bar0}
\nonumber\\&=&L_{r-\frac{1}{2}}G_\frac{1}{2}\mathbf{1}_{\bar0}-G_\frac{1}{2}L_{r-\frac{1}{2}}\mathbf{1}_{\bar0}
\nonumber\\&=&L_{r-\frac{1}{2}}\mathbf{1}_{\bar1}
-G_\frac{1}{2}\Big(\lambda^{r-\frac{1}{2}}(s+h_{r-\frac{1}{2}}(t))\Big)\mathbf{1}_{\bar0}
\nonumber \\&=&\lambda^{r-\frac{1}{2}}\Big(g_{r-\frac{1}{2}}(y)+\frac{1}{2}-h_{r-\frac{1}{2}}(y)\Big)\mathbf{1}_{\bar1}
\end{eqnarray}
and \begin{eqnarray}\label{4.636}
\nonumber&&(\frac{3}{4}-\frac{1}{2}r)G_r\mathbf{1}_{\bar1}
=[L_{r-\frac{1}{2}},G_\frac{1}{2}]\mathbf{1}_{\bar1}
\nonumber\\&=&L_{r-\frac{1}{2}}G_\frac{1}{2}\mathbf{1}_{\bar1}-G_\frac{1}{2}L_{r-\frac{1}{2}}\mathbf{1}_{\bar1}
\nonumber\\&=&L_{r-\frac{1}{2}}\Big(\lambda(t-\alpha)\Big)\mathbf{1}_{\bar0}
-G_\frac{1}{2}\Big(\lambda^{r-\frac{1}{2}}(x+g_{r-\frac{1}{2}}(y))\Big)\mathbf{1}_{\bar1}
\nonumber\\&=&\lambda^{r+\frac{1}{2}}\Big(\big(h_{r-\frac{1}{2}}(t)-g_{r-\frac{1}{2}}(t)+\frac{1}{2}\big)(t-\alpha)
-(r-\frac{1}{2})\big(t-(r-\frac{1}{2})\alpha\big)\Big)\mathbf{1}_{\bar0}.
\end{eqnarray}
By \eqref{4.635} and  \eqref{4.636}, we check
\begin{eqnarray*}
&&\lambda^{2r}\Big(\big(g_{r-\frac{1}{2}}(t)+\frac{1}{2}-h_{r-\frac{1}{2}}(t)\big)\big((h_{r-\frac{1}{2}}(t)
-g_{r-\frac{1}{2}}(t)+\frac{1}{2}\big)(t-\alpha)
\\&&-(r-\frac{1}{2})\big(t-(r-\frac{1}{2})\alpha\big)\big)\Big)\mathbf{1}_{\bar0}
\\&=&(\frac{3}{4}-\frac{1}{2}r)^2G_r^2\mathbf{1}_{\bar0}
=(\frac{3}{4}-\frac{1}{2}r)^2W_{2r}\mathbf{1}_{\bar0}
\\&=&\lambda^{2r}(\frac{3}{4}-\frac{1}{2}r)^2(t-2r\alpha)\mathbf{1}_{\bar0},
\end{eqnarray*}
which implies
 \begin{eqnarray}\label{4.8835}
\nonumber&&\big(g_{r-\frac{1}{2}}(t)+\frac{1}{2}-h_{r-\frac{1}{2}}(t)\big)\big((h_{r-\frac{1}{2}}(t)
-g_{r-\frac{1}{2}}(t)+\frac{1}{2}\big)(t-\alpha)
\nonumber\\&&~~~~~~~~~~~~~~~~~~~~~~~-(r-\frac{1}{2})\big(t-(r-\frac{1}{2})\alpha\big)\big)
=(\frac{3}{4}-\frac{1}{2}r)^2(t-2r\alpha).
\end{eqnarray}
Set  $\tau_r(t):=g_{r-\frac{1}{2}}(t)-h_{r-\frac{1}{2}}(t)$.  We rewrite  \eqref{4.8835} as
    \begin{eqnarray*}
(\tau_r(t)+\frac{1}{2})\Big((
\frac{1}{2}-\tau_r(t))(t-\alpha)
-(r-\frac{1}{2})(t-(r-\frac{1}{2})\alpha)\Big)
=(\frac{3}{4}-\frac{1}{2}r)^2(t-2r\alpha).
\end{eqnarray*}
Comparing  the coefficients of $t$ in above equation, we immediately obtain
$$(\tau_r(t)+\frac{1}{2})^2-\frac{3-2r}{2}(\tau_r(t)+\frac{1}{2})+(\frac{3-2r}{4})^2=0.$$
Thus $\tau_r(t)+\frac{1}{2}=\frac{3-2r}{4}$
 for $r\neq\frac{3}{2}$.
  From $[L_{1},G_\frac{1}{2}]\mathbf{1}_{\bar0}=0$, it is easy to get   $g_{1}(y)-h_{1}(y)=-\frac{1}{2}$. Then  $g_{m}(y)-h_{m}(y)=-\frac{m}{2}$ for $m\in\Z$.

 (ii)
 Using (i) in \eqref{4.635}, we obtain $G_r\mathbf{1}_{\bar 0}=\lambda^{r-\frac{1}{2}}\mathbf{1}_{\bar 1}$
 for any $r\in \frac{1}{2}+\Z$ and $r\neq\frac{3}{2}$.
Moreover, we check that
\begin{eqnarray*}
-\frac{3}{2}G_{\frac{3}{2}}\mathbf{1}_{\bar0}
=[L_2,G_{-\frac{1}{2}}]\mathbf{1}_{\bar0}
=L_2G_{-\frac{1}{2}}\mathbf{1}_{\bar0}-G_{-\frac{1}{2}}L_2\mathbf{1}_{\bar0}
=-\frac{3}{2}\lambda \mathbf{1}_{\bar1}.
\end{eqnarray*}
Hence, part (ii) also holds for $r=\frac{3}{2}$.

 By the similar discussion as  (ii),
 the (iii) clears.
\end{proof}
 \begin{lemm}\label{lem4.33}
For any $r \in \frac{1}{2}+\Z, h(t)\in\C[t],f(t,s)\in\C[t,s], g(y)\in\C[y],f(y,x)\in\C[y,x], $ we have
$$G_rf(t,s)\mathbf{1}_{\bar 0}=\lambda^{r-\frac{1}{2}}f(y,x-r)\mathbf{1}_{\bar 1},
\  G_rf(y,x)\mathbf{1}_{\bar 1}=\lambda^{r+\frac{1}{2}}(t-2r\alpha)f(t,s-r)\mathbf{1}_{\bar 0}.$$
\end{lemm}
\begin{proof}
According to $G_{r}L_0=(L_0-r)G_{r}$
and Lemma \ref{lem43334.33}, we show that
\begin{eqnarray*}
&&G_rf(t,s)\mathbf{1}_{\bar 0}= G_rf(W_0,L_0)\mathbf{1}_{\bar 1}
= f(W_0,L_0-r)G_r\mathbf{1}_{\bar 0}= \lambda^{r-\frac{1}{2}}f(y,x-r)\mathbf{1}_{\bar 1},
\\&&G_rf(y,x)\mathbf{1}_{\bar 1}= G_\frac{1}{2}f(W_0,L_0)\mathbf{1}_{\bar 1}
= f(W_0,L_0-r)G_r\mathbf{1}_{\bar 1}= \lambda^{r+\frac{1}{2}}(t-\alpha)f(t,s-r)\mathbf{1}_{\bar 0}.
\end{eqnarray*}
  The lemma holds.

\end{proof}
We present the main result of this section, which gives a
complete classification of free $U(\mathfrak{T})$-modules of rank $2$ over $\S^{\frac{1}{2}}$.

From \eqref{eqq4.1},  \eqref{eqq4.2},  \eqref{eqq4.3},  \eqref{eqq4.4} and Lemmas \ref{lem4.21452}, \ref{lem4.33},  we immediately get the following results.
\begin{theo}\label{th4.3}
 Assume that $M$ is an  $\S^{\frac{1}{2}}$-module such that the restriction of $M$ as a $U(\mathfrak{T})$-module is free of rank $2$. Up to a
parity-change,  we have $M\cong\Omega_{\S^{\frac{1}{2}}}(\lambda,\alpha,h)$ for $\lambda\in\C^*,\alpha\in\C,h(t)\in\C[t]$
with the $\S^{\frac{1}{2}}$-module structure defined as
in \eqref{3.22}-\eqref{3.88}.
\end{theo}

Based on Proposition \ref{prop2.6}, Theorem  \ref{th3.5} and Theorem \ref{th4.3}, we obtain the  corollary as follows.
\begin{coro} The category of free $U(\mathfrak{t})$-modules of rank $1$ over  $\S^0$ is
equivalent to the category of free $U(\mathfrak{T})$-modules of rank $2$ over
$\S^{\frac{1}{2}}$.
\end{coro}

\begin{remark} We note that     these modules are new classes of simple
modules over $\S^{\epsilon}$.
\end{remark}

\section*{Data availability}
The data that support the findings of this study are available from the corresponding author
upon reasonable request.

\section*{Acknowledgements}
This work was supported by the National Natural Science Foundation of China (No.12171129,  11971350),  and the Anhui Provincial  Natural Science Research Program for Institutions of Higher Learning (NO. KJ2018A0976).

{\scriptsize     Department of Mathematics, Jimei University, Xiamen, Fujian 361021, China    

{\it Email-address:} {\bf hypo1025@163.com}

\bigskip
School of Mathematical Sciences, Guizhou Normal University, Guiyang 550001, China

{\it Email-address:} {\bf daisheng158@126.com}

\bigskip
 School
of Statistics and Mathematics, Shanghai Lixin University of Accounting and Finance, Shanghai 201209, China

{\it Email-address:} {\bf lymaths@126.com}

\bigskip

  1. Department of Mathematics, Jimei University, Xiamen, Fujian 361021, China   2. School
of Mathematical Sciences, Tongji University, Shanghai 200092

{\it Email-address:} {\bf ycsu@tongji.edu.cn}
\bigskip
\end{document}